\documentclass[11pt]{amsart}
\usepackage[T1]{fontenc}
\usepackage[osf,sc]{mathpazo}
\usepackage{amsmath,amsthm,amsfonts,latexsym,amscd,amssymb,amsopn,enumerate,hyperref, mathrsfs}
\usepackage[utf8]{inputenc}
\usepackage{amsmath}
\usepackage{tikz-cd}
\usepackage{bm}
\usepackage{graphicx}
\usepackage{geometry}
\hypersetup{
    colorlinks=true,
    linktoc=all,
    linkcolor=blue,
    citecolor=red,
    filecolor=black,
    urlcolor=blue
}

\usepackage{geometry}
\usepackage[inline]{enumitem}
\makeatletter
\newcommand{\inlineitem}[1][]{%
    \ifnum\enit@type=\tw@
    {\descriptionlabel{#1}}
    \hspace{\labelsep}%
    \else
    \ifnum\enit@type=\z@
    \refstepcounter{\@listctr}\fi
    \quad\@itemlabel\hspace{\labelsep}%
    \fi} \makeatother
\newtheorem{thm}{Theorem}[subsection]
\newtheorem{lem}[thm]{Lemma}
\newtheorem{prop}[thm]{Proposition}

\newtheorem{cor}[thm]{Corollary}

\makeatletter
\def\namedlabel#1#2{\begingroup
    \def\@currentlabel{#2}%
    \label{#1}\endgroup
}
\makeatother

\theoremstyle{definition}

\theoremstyle{remark}
\newtheorem{remark}[thm]{Remark}
\numberwithin{equation}{subsection}
\begin{document}
\title[The Schur Indices and Degrees of Faithful Irreducible Representations]{Degrees of Faithful Irreducible Representations of Certain Metabelian Groups and a Question of Sim}
\author[Soham Swadhin Pradhan]{Soham Swadhin Pradhan}
\address{Soham Swadhin Pradhan, School of Mathematics, Harish-Chandra Research Institute, HBNI, Chhatnag Road, Jhunsi, Allahabad, 211019,  India. \, {\it Email address:} {\tt soham.spradhan@gmail.com}}

\author[B. Sury]{B. Sury{*}}
\address{B. Sury, Stat-Math Unit, Indian Statistical Institute, Bangalore Centre, 8-th Mile Mysore Road , Bangalore, 560059, India. \, {\it Email address:} {\tt surybang@gmail.com}}

\thanks{* Corresponding author.\\
E-mail addresses: surybang@gmail.com (B. Sury), soham.spradhan@gmail.com (Soham Swadhin Pradhan).}
\date{\sc \today}
\begin{abstract}
In this paper, we answer affirmatively a question of H S Sim on
representations in characteristic $0$, for a class of metabelian
groups. Moreover, we provide examples to point out that the
analogous answer is no longer valid if the solvable group has
derived length larger than 2. Let $F$ be a field of characteristic
$0$ and $\overline{F}$ be its algebraic closure. We prove that if
$G$ is a finite metabelian group containing a maximal abelian normal
subgroup which is a p-group with abelian quotient, all possible
faithful irreducible representations over $F$ have the same degree
and that the {\it Schur index} of any faithful irreducible
$\overline{F}$-representation with respect to $F$ is always $1$ or
$2$. H S Sim had proven such a result for metacyclic groups when the
characteristic of $F$ is positive and posed the question in
characteristic $0$. Our result answers this question for the above
class of metabelian groups affirmatively. We also determine
explicitly the Wedderburn component corresponding to any faithful
irreducible $\overline{F}$-representation in the group algebra
$F[G]$.
\end{abstract}

\maketitle

\section{Introduction}
Throughout this paper $G$ always denotes a finite group and $F$ a
field of characteristic $0$, unless it is explicitly stated. In this
article, we study the faithful irreducible $F$-representations of a
finite metabelian group with a prime power order maximal abelian normal subgroup cotaining the derived subgroup (or, equivalently abelian quotient). Some important classes of metabelian groups with a prime power order maximal abelian normal subgroup containing the derived subgroup are: holomorphs of cyclic groups of prime power orders (see \cite{Pradhan-2021}), metabelian $p$-groups.

In $1939$, Weisner (see \cite{Weisner-1939}) and in
$1954$, Gasch\"{u}tz  (see \cite{Gaschutz-1954}) characterized finite groups which possess a faithful irreducible representation. For metacyclic groups, an elementary characterization has been given by Sim (see \cite{Sim-2003}, Theorem $1.2$). For metabelian groups, a necessary and sufficient condition for the exixtence of a faithful irreducible representation has been given (see \cite{Bakshi-2013}, Theorem $3.1(2)$).

In $1993$, H. S. Sim studied the structure of faithful irreducible
representations of metacyclic groups, and he proved (see
\cite{Sim-1993}): {If $G$ is a metacyclic group and $F$ is a field
of positive characteristic, then all possible faithful irreducible
$F$-representations of $G$ have the same  degree}. Further, he
asked: {Is the result true when $F$ is of characteristic $0$?} The
question has a positive answer for nilpotent metacyclic groups of
odd order (see \cite{Sim-1995}). The question has an affirmative
answer for metabelian groups, over fields of {\it positive}
characteristic and also for finite groups having a cyclic quotient
by an abelian normal subgroup, over number fields (see
\cite{Rahul-2021}).

In this paper, we answer affirmatively Sim's question for a
metabelian group with a prime power order maximal abelian normal
subgroup containing the derived subgroup. Moreover, we provide
examples to point out that the analogous answer is no longer valid
if the solvable group has derived length $> 2$.

The Schur index was introduced by Schur in $1906$ (see \cite{Schur-1906}). Let $F$ be a field of characteristic $0$ and $\overline{F}$ the algebraic closure of $F$. Let $\rho$ be an irreducible $\overline{F}$-representation of $G$ with the corresponding character $\chi$. The character field $F(\chi)$ of $\rho$ over $F$ is an extension of $F$ adjoining all the values $\chi(g)$'s as $g$ varies in $G$. The {\it Schur index} $m_{F}(\rho)$ of $\rho$ with respect to $F$ is the minimal positive integer $m$ such that $m\chi$ is the character of an $F(\chi)$-representation of $G$ (see Section \ref{Section 2}).

The computation of the Schur index has been studied by many authors in the $20$-th century, some references are \cite{Feit-1982}, \cite{Feit-1983}, \cite{Alexandre-2000}, \cite{Unger-2017}. In $1930$, Brauer (see \cite{Richard-1930}) proved that every integer can occur as Schur index of some irreducible complex representation of some finite group. The most frequent situation is to have Schur indices equal to $1$ and $2$. In $1958$, Roquette (see \cite{Roquette-1958}) proved that the Schur index of any irreducible complex representation of a nilpotent group is always $1$ or $2$. In \cite{Alexandre-2001}, A. Turull gave an algorithm to calculate the Schur index of  irreducible complex representations of special linear groups, and they are all $1$ or $2$. In \cite{Gow-1981}, Gow showed that, in general, the Schur index of any irreducible complex representation of a finite classical group is always $1$ or $2$.  These results (and others) have led us to ask the question: {\it Is the Schur index of any faithful irreducible $\overline{F}$-representation with respect to $F$ of a metabelian group having a prime power order maximal abelian normal subgroup containing the derived subgroup always $1$ or $2$?} In this paper, we shall show that the question has an affirmative answer. Moreover, we provide an example to point out that the answer is not valid if a metabelian group does not have a prime power order maximal abelian normal subgroup cotaining the derived subgroup.
The main theorem of this paper is the following.
\begin{thm}\label{main}
Let $G$ be a finite metabelian group with a prime power order maximal abelian normal subgroup containing the derived subgroup (or, equivalently abelian quotient). Let $F$ be an arbitrary field of characterisic $0$ and $\overline{F}$ be its algebraic closure. Suppose that $\rho$, $\rho^{'}$ are two inequivalent faithful irreducible $\overline{F}$-representations of $G$ (if they exist). Let $m_{F}(\rho), m_{F}(\rho^{'})$ denote the Schur index of $\rho, \rho^{'}$ with respect to $F$. The following hold true.\\
(1) $m_{F}(\rho) = m_{F}(\rho^{'})$.\\
(2) All the faithful irreducible representations of $G$ over $F$ have the same degree.\\
(3) The Schur index of any faithful irreducible $\overline{F}$-representation of $G$ with respect to $F$ is always $1$ or $2$.\\
(4) If $K$ is any prime power order maximal abelian normal subgroup
of $G$ cotaining the derived subgroup and of exponent $p^n$, then
the common degree of faithful irreducible representations of $G$
over $F$ is equal to $[G:I]d$ or $2[G:I]d$, where $d$ is the common
degree of irreducible factors of $p^n$-th cyclotomic polynomial in
$F[X]$. Here, $d$ also equals the common degree of irreducible
representations of $K$ over $F$ with core-free kernel in $G$ and $I$
is the common inertia subgroup of all such representations of $K$ in
$G$.
\end{thm}
In \cite{Bakshi-2013}, a complete set of primitive central idempotents and the Wedderburn decomposition of the rational group algebra $\mathbb{Q}[G]$ of a metabelian group $G$ in terms of crossed product algebra has been given. In this paper, we explicitly determine the Wedderburn components corresponding to the faithful irreducible $\overline{F}$-representations of the group algebra $F[G]$, when $G$ is a metabelian group having a prime power order maximal abelian normal subgroup containing the derived subgroup.

For proving Theorem \ref{main}, we use, apart from the classical theories of Schur, Clifford and Yamada's work (see \cite{Yamada-1969}), the theory of central simple algebras and tools from algebraic number theory.

In Section \ref{Section 2}, we briefly recall the Schur index and two basic results on representations of abelian $p$-groups. In Section \ref{Section 3}, after recalling properties of crossed products, we construct certain cyclic algebras that will be used in the proof of the main theorem. In Section \ref{Section 4}, we give a description
how to compute the Wedderburn decomposition of the semisimple group algebra $F[G]$ from the Wedderburn decomposition of $\mathbb{Q}[G]$. After recalling some known facts on faithful irreducible representations of metabelian groups in Section \ref{Section 5}, we prove Theorem \ref{main} in Section \ref{Section 6}. In Section \ref{Section 7}, we expliclitly determine the Wedderburn components of $F[G]$, when $G$ is a finite metabelian group with a prime power order maximal abelian normal subgroup containing the derived subgroup, corresponding to faithful irreducible $\overline{F}$-representations. In last two sections \ref{Section 8} and \ref{Section 9}, we give examples illustrating our results and also point out two examples showing the results are not valid for some solvable groups of derived length $> 2$ or metabelian groups which do not have a prime power order maximal abelian normal subgroup containing the derived subgroup.
\subsection{Notations}
We set some notation that will be used throughout the paper.\\
$\bullet$ If $V$ is a $n$-dimensional vector space over a field $F$,
and $\rho:G\rightarrow {\rm GL}(V)$ is a group homomorphism, we say
that {\it $\rho$ is an $F$-representation of $G$} and $n$ is the
{\it degree of $\rho$}. \\
$\bullet$ For an irreducible $F$-representation $\rho$ of $G$,
ker$(\rho):= \{g \in G \,|\, \rho(g) = \rho(1)\}$.\\
$\bullet$ $G^{'}$ denotes the {\it derived subgroup} of $G$.\\
$\bullet$ For an element $x$ in a  group $G$, we denote by $o(x)$
its order.\\
$\bullet$ For a positive integer $n$, $C_{n}$ denotes a cyclic group of order $n$.\\
$\bullet$ For any $H \leq G$, the {\it $G$-core} of $H$ (i.e. the
intersection of all conjugates of $H$ in $G$) is the largest normal
subgroup of $G$ contained in $H$, and is denoted by core$_{G}{H}$;
$H$ is said to be core-free in $G$ if core$_{G}{H} = 1$.\\
$\bullet$ If $\eta$ is an $F$-representation of a subgroup $H$ of
$G$, the the {\it $G$-conjugate} of $\eta$ by $g$ is the representation
$\eta^{g}$ of $H^{g} = g^{-1}Hg$ given by $\eta^{g}(h) =
\eta(ghg^{-1}), h \in H^{g}$. \\
$\bullet$ If $\theta$ is an irreducible $F$-representation of a normal
subgroup $H$ of a group $G$, the {\it inertia subgroup} of $\theta$ in $G$ is
defined to be $I = \{g \in G : \theta^g = \theta \}$; it is the
stabilizer of $\theta$ under the $G$-action on the set $Irr_{F}(H)$
of irreducible $F$-representations of $H$.\\
$\bullet$ For a field $F$, by $F^*$ the multiplicative group of $F$,
and by $\overline{F}$ an algebraic closure of $F$. \\
$\bullet$ For a positive integer $n$, $\zeta_{n}$ denotes a
primitive $n$-th root of unity.\\
$\bullet$ We also use the abbreviation w.r.t. to denote `with
respect to'. \\
The rest of our notations are standard.

\section{Schur Index}\label{Section 2}
Let $G$ be a  finite group. Let $F$ be a field of characteristic $0$. Let $\widetilde \rho$ be an irreducible $\overline{F}$-representation of $G$ with the character $\widetilde \chi$. Then there exists a unique irreducible $F$-representation $\rho$ of $G$ such that $\widetilde \rho$ occurs as an irreducible constituent of $\rho \otimes_F \overline{F}$, with some multiplicity; the multiplicity is called the {\it Schur index} of $\widetilde \rho$ w.r.t. $F$, and is denoted by $m_F(\widetilde \rho)$.

The Schur index is also the index of a certain central simple algebra defined below.

The character $\widetilde{\chi}$ of $\widetilde{\rho}$ takes values in the field $F(\zeta_{u})$, where $u$ is the exponent of $G$. The character field $F(\widetilde \chi)$ of $\widetilde \rho$ over $F$ is an extension field of $F$ obtained by adjoining all the values $\widetilde{\chi}$'s as $g$ varies in $G$. In fact, $F(\widetilde{\chi})$ is an abelian Galois extension of $F$. Let $Z:= F(\widetilde{\chi})$. Then $e=(\widetilde{\chi}(1)/|G|)\sum_{g\in G} \widetilde{\chi}(g^{-1})g$ is a primitive central idempotent in $Z[G]$ corresponding to $\widetilde{\rho}$ and $A:=Z[G]e$ is a central simple algebra over $Z$ with the identity element $e$. By Wedderburn-Artin theorem (see \cite{Jacobson-2009}, Theorem $(3.21)$), $A\cong M_n(D)$ (as a $Z$-algebra) for a unique positive integer $n$ and a unique division algebra $D$ over $Z$. Then $\dim_{Z}(D)=m^2$ for some integer $m\ge 1$. The integer $m$ is called the {\it index} of the algebra $A$. It turns out that,  $m= m_F(\widetilde \rho)$. Thus, the Schur index of $\widetilde \rho$ w.r.t. $F$ is the index of the central simple algebra $Z[G]e$ and $e$ is the primitive central idempotent in $Z[G]$ corresponding to $\widetilde \rho$. For more details about Schur index, (see \cite{Curtis-1962}, $\S41A$ and $\S 70$,
\cite{Isaacs-1976}, Chapter $9$ and Chapter $10$, \cite{Reiner-1993}).
\subsection{Irreducible representations of abelian $p$-groups}
Consider a finite abelian $p$-group, $p$ a prime. Then, the
following lemmata are well known:

\begin{lem}\label{Lemma 2.0.1}
    Let $G$ be an abelian $p$-group and $F$ a field of characteristic $0$. Then every irreducible $F$-representation of $G$ factors through a faithful irreducible $F$-representation of a cyclic quotient.
\end{lem}
\begin{lem}\label{Lemma 2.0.2}
    Let $G$ be a cyclic group of order $p^{n}$ and $F$ a field of characteristic $0$. Then the set of all faithful irreducible $F$-representations of $G$ (up to equivalence) is in a bijective correspondence with the irreducible factors of the $p^{n}$-th cyclotomic polynomial in $F[X]$. Moreover, all the faithful irreducible $F$-representations of $G$ have the same degree, which is equal to the common degree of irreducible factors of the $p^{n}$-th cyclotomic polynomial over $F$.
\end{lem}
\section{Crossed Products and Cyclic Algebras}\label{Section 3}
In this section, we describe the construction of classical crossed products and cyclic algebras.
\subsection{Classical crossed product algebras} $L$ denotes a finite Galois extension of the field $F$,
with Galois group Gal$(L/F)$. Let $f: {\rm Gal}(L/F) \times {\rm Gal}(L/F) \longrightarrow L^{*}$ be a {\it factor set} (see \cite{Jacobson-2009}, $\S 8.4$) of ${\rm Gal}(L/F)$ with values in $L^{*}$. We may then form the {\it crossed product algebra} (see \cite{Jespers-2016}, page no. 66)
$$A = (L, \textrm{Gal}(L/F), f) = \displaystyle {\bigoplus_{\sigma \in {\rm Gal}(L/F)}}{Lu_{\sigma}},$$
having an $L$-basis consisting of symbols $\{u_{\sigma}: \sigma \in
{\rm Gal}(L/F)\}$, with multiplication in $A$ is defined by the
formulas below which give the action and the twisting respectively:
$$\mathrm{(action)} \hspace{2cm} u_{\sigma}x = \sigma(x)u_{\sigma},$$
$$\mathrm{(twisting)} \hspace{2cm} u_{\sigma}u_{\tau} = f(\sigma, \tau)u_{\sigma \tau},$$
for all $\sigma, \tau \in {\rm Gal}(L/F), x \in L$, and each $f(\sigma, \tau) \in L^{*}$. The condition $f$ be a factor set is precisely the condition that $A$ be an associative algebra. Note that $F$ is contained in the center of $A$ but that $L$
is not (if $L \neq K$), since the $u_{\sigma}$'s need not centralize $L$. Also note that $A$ is a central simple $F$-algebra with dim$_{F}{A} = n^2$, where $n = |{\rm{Gal}}(L/F)| = dim_{F}{L}$. When $f$ is the trivial factor set (all of whose values are $1$), the crossed product algebra $A$ is just a full matrix algebra $M_{n}(F)$.
\subsection{Cyclic algebras} Let $L/F$ be a cyclic (Galois) extension, i.e., a finite, normal, separable field extension with a cyclic Galois group, say generated by an automorphism $\sigma$ of order $n (= dim_{F}{L})$. Fix an element $a \in F^{*}$ and a symbol $u$, we let
$$A = \displaystyle {\bigoplus_{j = 0}^{n-1}{Lu^{j}}},$$
and multiply elements in $A$ by using the distributive law, and the two rules
$$u^{j}x = \sigma^{j}(x)u^{j},~ u^n = a,~\mathrm{for~all}~x \in L.$$
It is easy to see that $F \subseteq Z(A)$, so $A$ is an $F$-algebra, of dimension $n^{2}$. This algebra is denoted by $(L/F, \sigma, a)$, and is called the {\it cyclic algebra associated with $(L, \textrm{Gal}(L/F), \sigma)$ and $a \in F - \{0\}$} (see \cite{Jespers-2016}, page no. 70).
\begin{prop}[\cite{Jespers-2016}, Proposition 2.6.7]\label{cyclic}
Let $N_{L/F}$ denote the field norm from $L$ to $F$. $A \cong M_{n}(F)$ as an $F$-algebra if and only if $a \in N_{L/F}(L^{*})$.
\end{prop}
\subsection{Construction of certain cyclic algebras}\label{subsection 3.4}
We now construct certain cyclic algebras which will be used in the proof of our main theorem.

$(i)$ Let $p$ be odd. For $n \geq 2$ and let $\zeta_{p^{n}}$ be a primitive $p^{n}$-th root of unity. Let $F$ be a field of characteristic $0$. Suppose that $\zeta_{p^s} \in F$ for some positive integer $s$ such that $n > s \geq 1$. Let $L = F(\zeta_{p^n})$. As $p$ is odd and $\zeta_{p} \in F$, then $L/F$ is a cyclic (Galois) extension of degree a power of $p$, say, $p^l$. Let Gal$(L/F) = \langle \sigma \rangle \cong C_{p^l}$, where $\sigma$ is defined by: $\zeta_{p^n} \mapsto \zeta^{1 + p^{n-l}}_{p^n}$.
Fix the element $\zeta_{p^{s}} \in F^{*}$ and a symbol $u$, we let
$A = Lu^{0} \oplus Lu \oplus \dots \oplus Lu^{{p^l-1}},$
and multiply elements in $A$ by using the distributive law, and the two rules
$$ ux = \sigma(x)u, ~\mathrm{for~all}~  x\in  L
~\textrm{and}~
u^{p^l} = \zeta_{p^{s}} \in F^{*},$$
where we identify $u^{0}$ with the unity element of $A$.
Then $A$ is the cyclic algebra $(L/F, \sigma, \zeta_{p^{s}})$, which is a central simple $F$-algebra with a self-centralizing maximal subfield $L$.
%It is a well known fact that a cyclic algebra $(L/F, \sigma, a) \cong M_{n}(F)$ if and only if $a \in N_{L/F}(L)$ (Proposition $2.6.7$, \cite{Jespers-2016}, page no. $70$).
It can be shown that, there exists an element $\alpha \in L$ such that ${N}_{L/F}(\alpha) = \zeta_{p^{s}}$ in the following way.

Let $\zeta = \zeta_{p^{n}}$. Then
\begin{align*}
N_{L/F}(\zeta) = \zeta.\zeta^{1 + p^{n-l}} \cdots {\zeta^{({1 + p^{n-l}})}}^{p^l - 1}.
\end{align*}
The exponent
\begin{align*}
&1 + {1 + p^{n-l}}+ \cdots + {({1 + p^{n-l}})}^{p^l - 1}\\
& = \frac{(1 + p^{n-l})^{p^l} - 1}{p^{n-l}}\\
& = \sum^{p^l}_{r = 1}{{p^l}\choose{r}}\frac{p^{rn - rl}}{p^{n-l}}\\
& = p^{l} + \sum^{p^l}_{r = 2}{{p^l}\choose{r}}{p^{(r - 1)(n - l)}}.
\end{align*}
The basic observation is that for all $r \geq 2$,
\begin{align*}
(*)\hspace{.8cm} {{p^l}\choose{r}}{p^{(r - 1)(n - l)}} \equiv 0 ({\rm mod} p^n).
\end{align*}
To see this let $p^k$ be the power of $p$ dividing $r$ (includes the case $k = 0$). Then
\begin{align*}
{{p^l}\choose{r}} \equiv 0 ~~~({\rm mod}~ p^{l -k}).
\end{align*}
But $l - k + (r-1)(n-l) \geq n$ as $(r-2)(n-l) \geq k$. From the elementary observation $(*)$, it follows that
\begin{align*}
N_{L/F}(\zeta) = N_{L/F}(\zeta_{p^n}) = \zeta^{p^l}_{p^n} = \zeta_{p^{n-l}}.
\end{align*}
We may assume that $\zeta_{p^n}$ has been chosen so that $\zeta^{p^l}_{p^n} = \zeta_{p^{n-l}} = \zeta_{p^s}.$ Therefore, $\zeta_{p^s}$ belongs to the image of the norm map $N_{L/F}$. As $\zeta_{p^s}$ belongs to the image of the norm map $N_{L/F}$, then by Proposition \ref{cyclic}, $A = (L/F, \sigma, \zeta_{p^{s}}) \cong M_{p^l}(F)$, and of dimension $p^{2l}$ over $F$.

$(ii)$ Let $n \geq 3$ and $\zeta_{2^{n}}$ be a primitive $2^{n}$-th root of unity. Let $F$ be a field of characteristic $0$ and $L = F(\zeta_{2^n})$. Suppose that $F$ contains $\zeta_{2^s},$ where $2 \leq s < n$.  As $p = 2$ and $F$ contains $\sqrt{-1}$, then $L/F$ is a cyclic (Galois) extension, and of degree power of $2$, say, $2^l$. Let {Gal}$(L/F) = \langle \sigma \rangle \cong C_{2^l}$, where $\sigma$ is defined by $\zeta_{2^n} \longrightarrow \zeta^{1 + 2^{n-l}}_{2^n}$. Fix $\zeta_{2^{s}} \in F^{*}$ and a symbol $u$, we let
$A = Lu^{0} \oplus Lu \oplus \dots \oplus Lu^{{2^l-1}},$
and multiply elements in $A$ by using the distributive law, and the two rules
$$ux = \sigma(x)u, ~\mathrm{for~all}~  x\in  L
~\textrm{and}~
u^{2^l} = \zeta_{2^{s}} \in F^{*}.$$
Then $A$ form a cyclic algebra $(L/F, \sigma, \zeta_{2^{s}})$. It can be shown that there exists an element $\alpha \in L$ such that ${N}_{L/F}(\alpha) = \zeta_{2^{s}}$, and the proof is the same as the odd case. From the Proposition \ref{cyclic}, $A = (L/F, \sigma, \zeta_{2^{s}}) \cong M_{2^l}(F)$, and of dimension $2^{2l}$ over $F$.
\section{Wedderburn Decomposition of $F[G]$}\label{Section 4}
Let $G$ be a finite group and $F$ a field of characteristic $0$. Let $\chi$ be an irreducible $\overline{F}$-character of $G$.\\
%We fix the notation.\\
$e_{F}(\chi) :=$ the unique primitive central idempotent $e$ of $F[G]$ with $\chi(e) \neq 0$,\\
$A_{F}(\chi) := F[G]e_{F}(\chi) =$ the unique Wedderburn component $A$ of $F[G]$ with $\chi(A) \neq 0.$ $A_{F}(\chi)$ is called the {\it Wedderburn component} in $F[G]$ corresponding to $\chi$.

Let $\chi_{1}, \chi_{2}, \dots, \chi_{r}$ be representatives of the algebraically conjugacy classes of the absolutely irreducible characters of $G$ over $\mathbb{Q}$. Let $A_{i} = A_{\mathbb{Q}}(\chi_{i}), i = 1,2, \dots, r$.
Then the rational group algebra $\mathbb{Q}[G]$ is:
\begin{align*}
\mathbb{Q}[G] = A_{1} \oplus A_{2} \oplus \dots \oplus A_{r}.
\end{align*}
Note that for each $i$, $A_{i} \cong M_{n_{i}}(D_{i}),$ where $D_{i}$ is a finite dimensional division algebra over $\mathbb{Q}$. Therefore, we have
$\mathbb{Q}[G] \cong M_{n_{1}}(D_{1}) \oplus M_{n_{2}}(D_{2}) \oplus \dots \oplus M_{n_{r}}(D_{r}).$ It is well known that the center of $A_{i}$ is isomorphic to the the character field $\mathbb{Q}(\chi_{i})$. So the center of $D_{i}$ is isomorphic to $\mathbb{Q}(\chi_{i})$.

It is clear that
$$\displaystyle F[G] \cong \mathbb{Q}[G] \otimes_{\mathbb{Q}}{F}
\cong A_{1}\otimes_{\mathbb Q}{F} \oplus A_{2}\otimes_{\mathbb Q}{F} \oplus \dots \oplus A_{r}\otimes_{\mathbb Q}{F}.$$
    %&\cong M_{n_{1}}(D_{1})\otimes_{\mathbb Q}{F} \oplus M_{n_{2}}(D_{2})\otimes_{\mathbb Q}{F} \oplus \dots \oplus M_{n_{r}}(D_{r})\otimes_{\mathbb Q}{F}.

For notational convenience, set $\chi = \chi_{i}$, $A = A_{i}$. Then
$$A \otimes_{\mathbb{Q}}{F} \cong (A\otimes_{\mathbb{Q}(\chi)}{\mathbb{Q}(\chi)})\otimes_{\mathbb{Q}}{F} \cong A\otimes_{\mathbb{Q}(\chi)}({\mathbb{Q}(\chi)}\otimes_{\mathbb{Q}}{F}).$$
Let $E = \mathbb{Q}(\chi).$ Then
$$A \otimes_{\mathbb{Q}}{F} \cong A \otimes_{E}(E \otimes_{\mathbb{Q}}{F}).$$

As $E/\mathbb{Q}$ is a finite separable extension, by the Primitive Element Theorem $E = \mathbb{Q}(\alpha)$, for some $\alpha \in E$. Let $f =$ Min$_{\mathbb{Q}}(\alpha)$, the minimal polynomial of $\alpha$ over $\mathbb{Q}$. Then
\begin{align*}
E \cong \mathbb{Q}[X]/(f)~
\textrm{and}~ E \otimes_{\mathbb{Q}}{F} \cong \mathbb{Q}[X]/(f)\otimes_{\mathbb{Q}}{F} \cong (\mathbb{Q}[X]\otimes_{\mathbb{Q}}{F})/(f \otimes 1) \cong F[X]/{(f)}.
\end{align*}
As $E/\mathbb{Q}$ is separable extension, all the roots of $f$ in an algebraic closure of $\mathbb{Q}$ have multiplicity one. Therefore, we obtain
$f = f_{1}f_{2} \dots f_{k},$
where $f_{i}$'s are coprime irreducible polynomials in $F[X]$. Then by the Chinese Remainder Theorem, we have
\begin{align*}
E \otimes_{\mathbb{Q}}{F} \cong F[X]/(f) \cong \oplus^{k}_{i = 1}F[X]/(f_{i}),
\end{align*}
a direct sum of fields. In general, they may not be all isomorphic.

Assume that $F$ is a number field. Then $F/\mathbb{Q}$ is a normal extension.
%In this case, we now show that all of $F[X]/(f_{i})$'s are $\mathbb{Q}$-isomorphic.
Let $\overline{E}$ and $\overline{F}$ be algebraic closure of $E$ and $F$ respectively. Let $\alpha_{i} \in \overline{F}$ be a root of $f_{i}$ and set $K_{i} = F(\alpha_{i})$. Therefore,
\begin{align*}
E \otimes_{\mathbb{Q}}{F} \cong \oplus^{k}_{i = 1}K_{i}.
\end{align*}
For every $i = 1,2, \dots, k$, there is an $\mathbb{Q}$-isomorphism $\tau_{i}: E \longrightarrow \mathbb{Q}(\alpha_{i})$ such that $\tau_{i}(\alpha) = \alpha_{i}$. Since $\overline{E}$ and $\overline{F}$ are the algebraic closure of $\mathbb{Q}$, they are isomorphic. Each $\tau_{i}$ extends to an isomorphism $\tau_{i} : \overline{E} \longrightarrow \overline{F}$, which keep denoting $\tau_{i}$. Then
$\sigma_{i} = \tau_{i} o \tau^{-1}_{1} \in$ Gal$(\overline{E}/\mathbb{Q})$ and $\sigma_{i}(\mathbb{Q}(\alpha_{1})) = \mathbb{Q}(\alpha_{i})$. Since $F/\mathbb{Q}$ is a normal extension, it follows that
$$\sigma_{i}(K_{1}) = \sigma_{i}(F(\alpha_{1})) = F(\alpha_{i}) = K_{i}.$$
Thus $\sigma_{i}$ restricts to an $\mathbb{Q}$-isomorphism $K_{1} \longrightarrow K_{i}.$
Observe that $K_{i} = F(\alpha_{i})$ is the compositum of $\tau_{i}(E)$ and $F$ in $\overline {\mathbb{Q}}$. Using this we can consider $K_{i}$ an $(E, F)$-bimodule. This defines a structure of $E$-algebra and of $F$-algebra in $K_{i}$. Clearly $\sigma_{i}: K_{1} \longrightarrow K_{i}$ is an isomorphism $(E, F)$-bimodules and hence it is an isomorphism of $E$-algebras and of $F$-algebras. We have
\begin{align*}
A \otimes_{\mathbb{Q}}{F} \cong A \otimes_{E}(E \otimes_{\mathbb{Q}}{F}) \cong \oplus^{k}_{i = 1} (A \otimes_{E} K_{i}).
\end{align*}
As $A$ is a central simple $E$-algebra, the algebra $A \otimes_{E}{K_{i}}$ is a central simple $K_{i}$-algebra.
%Since $A$ is finite dimensional over $\mathbb{Q}$, $A \otimes_{E}{K_{i}}$ is also a finite dimensional over $K_{i}$.
As all the $K_{i}$'s are isomorphic as $E$-algebras and therefore all the $C_{i} = A \otimes_{E}{K_{i}}$ are isomorphic $E$-algebras. Furthermore, the $F$-algebra structure on $K_{i}$ induces a structure of $F$-algebra on $C_{i}$ and the isomorphism $C_{1} \longrightarrow C_{i}$ is an isomorphism of $F$-algebras. So we have
$$A \otimes_{\mathbb{Q}}{F} \cong \oplus^{k}_{i = 1}C_{i} ~~~~~(\rm{as}~F-\rm{algebras}).$$
%Note that $C_{i}$'s are all $F$-isomorphic.
%As a consequence of that the indices of these $F$-algebras are the same.

If $F$ is an arbitrary field of characteristic $0$ then we have
$$A \otimes_{\mathbb{Q}}{F} \cong A \otimes_{\mathbb{Q}(\chi)}(\mathbb{Q}(\chi)\otimes_{\mathbb{Q}}{F}) \cong A \otimes_{E}{F(\chi)} \oplus \dots \oplus A \otimes_{E}{F(\chi)},$$
where the last summands correspond respectively to the algebraic conjugate characters $\chi^{\sigma}, ~\sigma \in$ Gal$(\mathbb{Q}(\chi)\cap F/\mathbb{Q})$.
%So the total number of summands is equal to $\frac{[\mathbb{Q}(\chi):\mathbb{Q}]}{[F(\chi): F]}$.
Therefore, we have
$$A \otimes_{\mathbb{Q}}{F} \cong [\mathbb{Q}(\chi)\cap F: \mathbb{Q}] A \otimes_{E}{F(\chi)}.$$
\begin{remark}\label{Remark 4.0.1}
If $F$ is an arbitrary field of characteristic $0$ then one can determine the Wedderburn decomposition of $F[G]$ from the Wedderburn decomposition of $\mathbb{Q}[G]$. If $\chi$ is an irreducible $\overline{F}$-character of $G$ then
$$A_{F}(\chi) \cong A_{\mathbb{Q}}(\chi) \otimes_{\mathbb{Q}(\chi)}{F(\chi)}~ \textrm{and}~ A_{\mathbb{Q}}(\chi) \otimes_{\mathbb{Q}}{F} \cong [\mathbb{Q}(\chi)\cap F: \mathbb{Q}] A_{F}(\chi).$$
\end{remark}
\subsection{Generalized Quaternion Groups, Dihedral Groups and Semidihedral Groups}
It is a well-known fact that Generalized Quaternion Group, Dihedral Group and Semidihedral Group all have a unique faithful irreducible $\mathbb{Q}$-representation.

The {\it generalized quaternion group} $Q_{2^n}, n \geq 3$, of order $2^n$, is defined by
$$Q_{2^n} = \langle{ x, y ~|~ x^{2^{n-1}} = 1, y^2 = x^{2^{n-2}}, y^{-1}xy = x^{-1}}\rangle.$$

The {\it dihedral group} $D_{2^n}, n \geq 2$, of order $2^n$, is defined by
$$D_{2^n} = \langle{ x, y ~|~ x^{2^{n-1}} = y^2 = 1, y^{-1}xy = x^{-1}}\rangle.$$

The {\it semidihedral group} $SD_{2^n}, n \geq 4$, of order $2^n$, is defined by
$$SD_{2^n} = \langle{ x, y ~|~ x^{2^{n-1}} = y^2 = 1, y^{-1}xy = x^{-1 + 2^{n-2}}}\rangle.$$

The following results are well-known, and we skip the proofs.

We set $\zeta = \zeta_{2^{n-1}}.$
\begin{lem}\label{Lemma 4.1.1}
Let $\rho$ be a faithful irreducible complex representation of $Q_{2^n}$ and $\chi$ be its corresponding character. Then
$\mathbb{Q}(\chi) = \mathbb{Q}(\zeta + \zeta^{-1})$, and
$m_{\mathbb{Q}}(\rho) = 2$. The simple component of $\mathbb{Q}[G]$ corresponding to $\chi$ is the cyclic algebra: $$(\mathbb{Q}(\zeta)/\mathbb{Q}(\zeta + \zeta^{-1}), \sigma, -1) = \mathbb{Q}(\zeta).1 + \mathbb{Q}(\zeta).u, u^2 = -1, u \zeta u^{-1} = \zeta^{\sigma} = \zeta^{-1},$$ which is central simple algebra over $\mathbb{Q}(\zeta + \zeta^{-1})$. Moreover, this algebra is Hamilton's quaternion algebra over $\mathbb{Q}(\zeta + \zeta^{-1})$.
\end{lem}
\begin{lem}\label{Lemma 4.1.2}
Let $\rho$ be a faithful irreducible complex representation of $D_{2^n}$ and $\chi$ be its corresponding character. Then $\mathbb{Q}(\chi) = \mathbb{Q}(\zeta + \zeta^{-1}),$ and $m_{\mathbb{Q}}(\rho) = 1$.
The simple component in $\mathbb{Q}[G]$ corresponding to $\chi$ is the cyclic algebra: $$(\mathbb{Q}(\zeta)/\mathbb{Q}(\zeta + \zeta^{-1}), \sigma, 1) = \mathbb{Q}(\zeta).1 + \mathbb{Q}(\zeta).u, u^2 = 1, u \zeta u^{-1} = \zeta^{\sigma} = \zeta^{-1},$$ which is a central simple algebra over $\mathbb{Q}(\zeta + \zeta^{-1})$. This algebra splits, and is isomorphic to $M_{2}(\mathbb{Q}(\zeta + \zeta^{-1})$.
\end{lem}
\begin{lem}\label{Lemma 4.1.3}
Let $\rho$ be a faithful irreducible complex representation of $SD_{2^n}$ and $\chi$ be its corresponding character. Then $\mathbb{Q}(\chi) = \mathbb{Q}(\zeta -  \zeta^{-1}),$ and
$m_{\mathbb{Q}}(\rho) = 1$. The simple component in $\mathbb{Q}[G]$ corresponding to $\chi$ is the cyclic algebra: $$(\mathbb{Q}(\zeta)/\mathbb{Q}(\zeta - \zeta^{-1}), \tau, 1) = \mathbb{Q}(\zeta).1 + \mathbb{Q}(\zeta).u, u^2 = 1, u \zeta u^{-1} = \zeta^{\tau} = \zeta^{-1 + 2^{n-2}},$$ which is a central simple algebra over $\mathbb{Q}(\zeta - \zeta^{-1})$. This cyclic algebra splits and is isomorphic to $M_{2}(\mathbb{Q}(\zeta - \zeta^{-1})$.
\end{lem}
\begin{remark}\label{Remark 4.1.4}
Let $F$ be any arbitrary field of characteristic $0$. The following hold true.
\begin{enumerate}
\item[(i)] If $G = {Q}_{2^{n}}, n \geq 3$, and $\rho$ is a faithful irreducible $\overline{F}$-representation of $G$ with the character $\chi$, then
$A_{F}(\chi) \cong \mathbb{H}(\mathbb{Q}(\zeta + \zeta^{-1}))\otimes_{\mathbb{Q}(\zeta + \zeta^{-1})}{F(\zeta + \zeta^{-1})}.$
\item[(ii)] If $G = {D}_{2^{n}}, n \geq 2$, and $\rho$ is a faithful irreducible $\overline{F}$-representation of $G$ with the character $\chi$, then
$A_{F}(\chi) \cong M_{2}(\mathbb{Q}(\zeta + \zeta^{-1}))\otimes_{\mathbb{Q}(\zeta + \zeta^{-1})}{F(\zeta + \zeta^{-1})}.$
\item[(iii)] If $G = {SD}_{2^{n}}, n \geq 4$, and $\rho$ is a faithful irreducible $\overline{F}$-representation of $G$ with the character $\chi$, then
$A_{F}(\chi) \cong M_{2}(\mathbb{Q}(\zeta - \zeta^{-1}))\otimes_{\mathbb{Q}(\zeta - \zeta^{-1})}{F(\zeta - \zeta^{-1})}.$
\end{enumerate}
\end{remark}
\section{Faithful Irreducible Representations of Metabelian Groups}\label{Section 5}
In this section, we state some results on faithful irreducible
representations of metabelian groups that are used in the proofs of
Theorem \ref{main}. These results are proved in \cite{Rahul-2021}.

\noindent{\it Hypothesis}: {\it $G$ is a metabelian group, $K$ is a maximal abelian normal subgroup of $G$ with $G/K$ abelian (i.e. $K$ is a maximal abelian normal subgroup containing $G^{'}$), and $F$ is a field of characteristic $0$.}
\begin{thm}[\cite{Rahul-2021}, Theorem 3.2]\label{thm4.2}
Let $\rho$ be a faithful irreducible $F$-representation of $G$, $\eta$ an irreducible constituent of $\rho \downarrow^{G}_{K}$, and $I$ the inertia subgroup of $\eta$ in $G$. Then there exists an irreducible $F$-representation $\theta$ of $I$ such that $\theta \uparrow^{G}_{I} \cong \rho$ and $\theta \downarrow^{I}_{K} \cong \eta^{\oplus m}$, the direct sum of $m$ copies of $\eta$, for some $m\ge 1$. Moreover, $m$ is equal to the Schur index (w.r.t. $F$) of any irreducible constituent of $\theta\otimes_{F}{\overline{F}}$.
\end{thm}
The proof is based on Clifford's theory, Schur's theory of group
representations.

The following easy lemma is useful to prove Theorem \ref{main}.
\begin{lem} [\cite{Rahul-2021}, Lemma 3.1]\label{lemma3.1}
If $L$ is a subgroup of $K$ such that $K/L$ is cyclic and ${\rm
core}_{G}{L} = 1$, then
\begin{itemize}
\item[(i)] $C_{G}{(K/L)} = K$;
\item[(ii)] $|K/L| = {\rm exp}(K)$.
\end{itemize}
\end{lem}

\begin{remark}[\cite{Rahul-2021}, Remark 3.3]\label{remark3.3}
In Theorem \ref{thm4.2}, if $F=\overline{F}$, then $\eta$ is linear. Hence $I$ (the inertia subgroup of $\eta$) is $K$. Consequently, $\theta=\eta$ and $\rho\cong\eta\uparrow_K^G$. Thus,  over $\overline{F}$, a faithful irreducible representation of $G$ (if it exists) is induced from a linear representation of $K$.
\end{remark}
\begin{thm}[\cite{Rahul-2021}, Theorem 3.4]\label{theorem3.4}
If $\eta$ and $\eta'$ are irreducible $F$-representations of $K$ with core-free kernels in $G$, then they have the same inertia subgroup in $G$.
\end{thm}
The proof follows from the proof of Theorem \ref{thm4.2}, and the
duality between the subgroups of $\textrm{Hom}(K, \overline{F}^*)$
and $K$ (see \cite{Huppert-1967}, Theorem $6.4$(e))

\section{Proof of Theorem \ref{main}}\label{Section 6}
In this section, we prove the main theorem using the following
theorem of Yamada.
\begin{thm}[\cite{Yamada-1969}, Theorem 2]\label{Theorem 7.0.1}
Let $G$ be a finite group and $K$ a normal subgroup of $G$.
Let $F$ a field of characteristic $0$.
Let $\chi$ be an irreducible $\overline{F}$-character of $G$. Suppose that $\chi = \lambda \uparrow^{G}_{K}$, where $\lambda$ is a linear character of $K$. Define
$$I:=\{ g\in G\,: \,  \mbox{there exists}~\tau(g)\in {\rm Gal}(F(\lambda)/F)~ \mbox{such that}~ \lambda^g=\lambda^{\tau(g)}\}.$$ Let $f_{1}K, f_{2}K, \dots, f_{n}K$ be all the distinct left cosets of $K$ in $I$, and $f_{i}f_{j} = k_{ij}f_{\nu(i,j)}, k_{ij} \in K$. Set $\tau(f_{i}) = \tau_{i}$ and $\beta(\tau_{i}, \tau_{j}) = \lambda(k_{ij}), 1 \leq i, j \leq n$. The following properties hold.
\begin{itemize}
\item[(1)] $\psi = \lambda \uparrow^{I}_{K}$ is an irreducible $\overline{F}$-character of $I$.
\item[(2)] $I/K \cong \{\tau_{1}, \tau_{2}, \dots, \tau_{n}\} \cong \textrm{Gal}(F(\lambda)/F(\psi)) = \textrm{Gal}(F(\lambda)/F(\chi))$.
\item[(3)] $\beta$ is a factor set of $\textrm{Gal}(F(\lambda)/F(\psi)) = {\textrm Gal}(F(\lambda)/F(\chi))$ consisting roots of unity.
\item[(4)] The Wedderburn component $A_{F}(\psi)$ of $F[I]$, which corresponds to $\psi$ is a crossed product algebra $(F(\lambda), \textrm{Gal}(F(\lambda)/F(\psi)), \beta)$.
\item[(5)] $m_{F}(\lambda \uparrow^{G}_{K}) = m_{F}(\lambda \uparrow^{I}_{K})$.
\item[(6)] If $A_{F}(\psi) \cong M_{r}(D)$, for a finite dimentonal division algebra $D$ over $F$ and for an integer $r$, then $A_{F}(\chi) \cong M_{rt}(D),$ $t = [G : I]$.
\end{itemize}
\end{thm}

In addition, we recall the following results from the book [20]
which are also used in our proof.\\ For a prime $p$ we denote by
$v_{p}(n)$ the valuation of $n$ at $p$; that is  the maximum $p$-th
power $p^{v_{p}(n)}$ dividing $n$.
\begin{lem}[\cite{Jespers-2016}, Lemma 13.5.2]\label{Lemm 5.2}
Let $G$ be a finite $p$-group with a maximal abelian subgroup which is cyclic and normal in $G$. Then $G$ is isomorphic to the one of the following groups:\\
$P_{1} = \langle{a, b ~|~ a^{p^n} = b^{p^k} = 1, b^{-1}ab = a^{r}}\rangle,$
with either $v_{p}(r-1) = n - k$ or $p = 2$ and $r \not \equiv 1$ {\rm mod} $4$,\\
$P_{2} = \langle{a, b, c ~|~ a^{2^n} = b^{2^k} = c^2 = 1, bc = cb, b^{-1}ab = a^{r}, c^{-1}ac = a^{-1}}\rangle,$ with $r \equiv 1$ {\rm mod} $4$,\\
$P_{3} = \langle{a, b, c ~|~ a^{2^n} = b^{2^k} = 1, c^2 = a^{2^{n-1}}, bc = cb, b^{-1}ab = a^{r}, c^{-1}ac = a^{-1}}\rangle,$
with $r \equiv 1$ {\rm mod} $4$.
\end{lem}
The two lemmata below can be deduced readily from the proof of
Theorem 13.5.3 in \cite{Jespers-2016}.
\begin{lem}\label{Lemma 6.0.2}
Let $G = P_{2}$. If $\rho$ is a faithful irreducible
$\mathbb{Q}$-representation of $G$ then the corresponding Wedderburn
component in $\mathbb{Q}[G]$ is isomorphic to $M_{2^{k +
1}}(\mathbb{Q}(\zeta_{2^{n-k}} + \zeta^{-1}_{2^{n-k}}))$.
\end{lem}
\begin{lem}\label{Lemma 5.5}
Let $G = P_{3}$. If $\rho$ is a faithful irreducible
$\mathbb{Q}$-representation of $G$ then the corresponding Wedderburn
component in $\mathbb{Q}[G]$ is isomorphic to
$M_{2^{k}}(\mathbb{H}(\mathbb{Q}(\zeta_{2^{n-k}} +
\zeta^{-1}_{2^{n-k}})))$, where
$\mathbb{H}(\mathbb{Q}(\zeta_{2^{n-k}} + \zeta^{-1}_{2^{n-k}}))$ is
the quaternion division algebra over $\mathbb{Q}(\zeta_{2^{n-k}} +
\zeta^{-1}_{2^{n-k}})$.
\end{lem}
\begin{remark}
If $G = P_{2}$ (resp. $G = P_{3}$) and $F$ has characteristic $0$, then the Schur index any faithful irreducible $\overline{F}$-representation of $G$ w.r.t. $F$ equals $1$ (resp. either $1$ or $2$).
\end{remark}
Now, we can prove our main result Theorem \ref{main}.

{\it Hypothesis.} Let $G$ be a metabelian group with a prime power order maximal abelian normal subgroup $K$ cotaining the derived subgroup $G^{'}$, i.e. $G/K$ is abelian. Let exp$(K) = p^{n}$, $p$ a prime and $n$ a positive integer. Let $F$ be an arbitrary field of characteristic $0$. Suppose that $\rho$, $\rho^{'}$ are any two faithful irreducible $\overline{F}$-representations of $G$.

{\bf Proof of Theorem \ref{main} (1).}
We show that $m_{F}(\rho) = m_{F}(\rho^{'})$. Without loss of generality, we can assume that $G$ is non-abelian, and therefore $G/K\neq \langle 1 \rangle$. From Remark
\ref{remark3.3}, $\rho$ and $\rho'$ are induced from linear
$\overline{F}$-representations of $K$, say $\lambda$ and $\lambda'$
respectively. Let $\chi$, $\chi'$ be the characters of $\rho$ and
$\rho'$ respectively. From the proof of Theorem \ref{thm4.2} (see
\cite{Rahul-2021}, Theorem 3.2), $\ker\lambda$ and $\ker\lambda'$
are core-free in $G$, and then by Lemmata \ref{Lemma 2.0.1},
\ref{Lemma 2.0.2} and \ref{lemma3.1}$(ii)$, it is easy to see that
$F(\chi)=F(\chi')$ and $F(\lambda)=F(\lambda') = F(\zeta_{p^n})$. As
$m_F(\rho)=m_{F(\chi)}(\rho)$,
$m_F(\rho^{'})=m_{F(\chi^{'})}(\rho^{'})$ (see \cite{Isaacs-1976},
Corollary $(10.2)(a)$) and $F(\chi)=F(\chi')$, then without loss of
generality, we can assume that $F=F(\chi)=F(\chi')$.

Let $I:=\{ g\in G\,: \,  \mbox{there exists}~\tau(g)\in {\rm
Gal}(F(\lambda)/F)~ \mbox{such that}~ \lambda^g=\lambda^{\tau(g)}\}$
and $I^{'}:=\{ g\in G\,: \,  \mbox{there exists}~\tau(g)\in {\rm
Gal}(F({\lambda^{'}})/F)~ \mbox{such that}~
{\lambda^{'}}^g={\lambda^{'}}^{\tau(g)}\}.$ As
$F(\lambda)=F(\lambda')$,  from Theorem \ref{Theorem 7.0.1} of
Yamada recalled above it follows that
$$(*)\hspace{.8cm} I/K\cong {\rm Gal}(F(\lambda)/F) = {\rm Gal}(F(\lambda')/F)\cong I'/K.$$

If $\eta:=\oplus_{\tau} \lambda^{\tau}$, where $\tau$ runs over
${\rm Gal}(F(\lambda)/F)$, then $\eta$ is an irreducible
$F$-representation of $K$; in this case, $I$ is the inertia subgroup
of $\eta$ in $G$, and $\ker\eta=\ker\lambda$, which is core-free in
$G$. Similarly, $I'$ is the inertia subgroup of an irreducible
$F$-representation $\eta^{'}$ of $K$,  and
$\ker\eta^{'}=\ker\lambda^{'}$, which is core-free in $G$. By
Theorem \ref{theorem3.4},  we have $I=I'$. Also, from Theorem \ref{Theorem 7.0.1}, we have
$m_F(\lambda\uparrow_K^G)=m_F(\lambda\uparrow_K^I)$ and
$m_F(\lambda'\uparrow_K^G)=m_F(\lambda'\uparrow_K^I)$. Therefore,
without loss of generality, we may assume that $I = G$ and $I^{'} = G$. Again since $I = I^{'}$, without loss of generality, we can assume that $I = I^{'} = G$.
As ker$\lambda \triangleleft I$, ker$\lambda^{'} \triangleleft I^{'}$
and $I= I^{'} = G$, this implies that ker$\lambda = $ ker$\lambda^{'} = 1$. Thus $G$ reduces to a metabelian group with a maximal abelian normal subgroup $K$ of order $p^{n}$ which is cyclic. From the equation $(*)$, we have
$$(**)\hspace{.8cm} G/K \cong {\rm Gal}(F(\lambda)/F).$$

    We consider the cases $p$ odd and $p = 2$ separately.

Write $E$ for the field $F(\lambda)=F(\lambda') = F(\zeta_{p^n})$.

{\bf Case $p$ odd.}

It is easy to see that ${\rm Gal}(E/F)$ is isomorphic to a normal subgroup of Gal$(\mathbb{Q}(\zeta_{p^n})/\mathbb{Q})$. Therefore ${\rm Gal}(E/F)$ is isomorphic to a cyclic subgroup of Aut$(C_{p^n}) (\cong C_{p^{n-1}(p-1)})$. From the equation $(**)$, $G/K$ is cyclic, let $ G/K = \langle yK\rangle$, for some $y$ in $G$. From the definition of $I$, there exists an automorphism $\sigma$ in ${\rm Gal}(F(\lambda)/F)$ such that $\lambda^y=\lambda^{\sigma}.$ It is also easy to see that $\lambda'^y=\lambda'^{\sigma}$.

Let $G/K$ be of order $d$. Then $y^d$ is a central element in $G$, and so $\langle {y^d}\rangle \subseteq Z(G)$.
Notice that $\lambda$ and $\lambda'$ are  faithful on $\langle y^d\rangle$. If $o(y^d) = p^s$, then without loss of generality, we can assume that
\begin{equation}\label{equation4.1}
\lambda(y^d)=\zeta_{p^s} \hskip5mm \mbox{ and } \hskip5mm \lambda'(y^d)=\zeta_{p^s}^j \hskip5mm \mbox{ for some } (j, p^s)=1.
\end{equation}
Now $e=(\chi(1)/|G|)\sum_{g\in G} \chi(g^{-1})g$ and
$e'=(\chi'(1)/|G|)\sum_{g\in G} \chi'(g^{-1})g $ are the primitive
central idempotents in $F[G]$ corresponding to $m_F(\rho)\chi$ and
$m_F(\rho')\chi'$ respectively. Hence $F[G]e$ and $F[G]e'$ are
central simple algebras over $F$, and from Theorem \ref{Theorem 7.0.1}, they are isomorphic to crossed product algebras
$(E, {\rm Gal}(E/F),\beta)$ and $(E, {\rm Gal}(E/F),\beta')$ (see Section \ref{Section 3})
respectively, where (using the expression \ref{equation4.1}), the
factor sets $\beta$ and $\beta'$ are given by the expressions
\begin{align*}\tag{***}\label{equation5.3}
\beta(\sigma^a,\sigma^b)= 1, & &\beta'(\sigma^a,\sigma^b)=1  \hskip5mm\mbox{ if } a+b<d,\\
\mbox{ and }\hskip3mm \beta(\sigma^a,\sigma^b)=\lambda(y^d)=\zeta_{p^s},  & &
\beta'(\sigma^a,\sigma^b)=\lambda'(y^d)=\zeta_{p^s}^{j}  \hskip5mm \mbox{ if } a+b\ge d.
\end{align*}
%Let $[\beta]$ (resp. $[\beta']$) denotes the cohomology class of $\beta$ (resp. $\beta'$) in $H^2({\rm Gal}(E/F), E^*)$.
Now $m_F(\rho)$ is the index of $F$-algebra $F[G]e$ $\cong$ $ (E, {\rm Gal}(E/F),\beta)$ (see Section \ref{Section 2}).
Similarly, $m_F(\rho^{'})$ is the index of $F$-algebra $F[G]e^{'}$ $\cong$ $ (E, {\rm Gal}(E/F),\beta^{'})$ (see Section \ref{Section 2}).
The crossed product algebra $(E, {\rm Gal}(E/F),\beta)$ is isomorphic to the cyclic algebra $(E/F, \sigma, \zeta_{p^s})$ (see Section \ref{Section 4}),
and similarly, $(E, {\rm Gal}(E/F),\beta^{'})$ is isomorphic to the cyclic algebra $(E/F, \sigma, \zeta_{p^s}^{j})$ (see Section \ref{Section 4}).
Now this discussion now splits into two subcases.\\\\
{\it Subcase I. The order of $y^{d}$ = 1:} In this case, it is clear that $(E, {\rm Gal}(E/F),\beta)$ and $(E, {\rm Gal}(E/F),\beta')$ have the same index $1$.
Hence, $m_F(\rho)=m_F(\rho') = 1$.\\\\
{\it Subcase II. The order of $y^{d} = p^s, s \geq 1$:} In this case, notice that $\zeta_{p^s} \in F$. If $E/F$ is a cyclic Galois extension of degree $p^l$, say, then the cyclic algebra $(E/F, \sigma, \zeta_{p^s})$ is isomorphic to the matrix algebra $M_{p^l}(F)$ (see Subsection \ref{subsection 3.4} $(i)$). Similarly, the cyclic algebra $(E/F, \sigma^{'}, \zeta^{j}_{p^s})$ is isomorphic to the matrix algebra $M_{p^l}(F)$. Thus, $(E, {\rm Gal}(E/F),\beta)$ and $(E, {\rm Gal}(E/F),\beta')$ have the same index $1$. Hence, $m_F(\rho)=m_F(\rho') = 1$.

{\bf Case $p = 2$.} This discussion now splits into two subcases.

    {\it Subcase I. $F$ contains $\sqrt{-1}$.} As $p = 2$ and $F$ contains $\sqrt{-1}$, then ${\rm Gal}(E/F) = {\rm Gal}(F(\zeta_{2^{n}})/F)$ is a cyclic $2$-group, say, of order $2^l$.

Similarly to the case of odd $p$, we may proceed as follows. From the equation $(**)$, ${\rm Gal}(E/F)\cong G/K$, let $ G/K = \langle yK\rangle$, for some $y$ in $G$. From the definition of $I$, there exists an automorphism $\sigma$ in ${\rm Gal}(E/F)$ such that $\lambda^y=\lambda^{\sigma}$. It is easy to see that $\lambda'^y=\lambda'^{\sigma}$.\\
As $G/K$ is of order $2^l$, then $y^{2^l}$ is a central element in $G$, and so $\langle {y^{2^l}}\rangle \subseteq Z(G)$.
Notice that $\lambda$ and $\lambda'$ are  faithful on $\langle y^{2^l}\rangle$. If $o(y^{2^l}) = 2^s$, then without loss of generality, we can assume that
\begin{equation}\label{equation4.2}
\lambda(y^{2^l})=\zeta_{2^s} \hskip5mm \mbox{ and } \hskip5mm \lambda'(y^{2^l})=\zeta_{2^s}^j \hskip5mm \mbox{ for some } (j, 2^s)=1.
\end{equation}
Now $e=(\chi(1)/|G|)\sum_{g\in G} \chi(g^{-1})g$ and $e'=(\chi'(1)/|G|)\sum_{g\in G} \chi'(g^{-1})g $ are the primitive central idempotents in $F[G]$ corresponding to $m_F(\rho)\chi$ and $m_F(\rho')\chi'$ respectively. Hence $F[G]e$ and $F[G]e'$ are central simple algebras over $F$, and by Theorem \ref{Theorem 7.0.1}, they are isomorphic to crossed products $(E, {\rm Gal}(E/F),\beta)$ and
$(E, {\rm Gal}(E/F),\beta')$ respectively (see Section \ref{Section 3}), where (using equation \ref{equation4.2}), the factor sets $\beta$ and $\beta'$ are given by
\begin{align*}\tag{***}\label{equation5.3}
\beta(\sigma^a,\sigma^b)= 1, & &\beta'(\sigma^a,\sigma^b)=1  \hskip5mm\mbox{ if } a+b<2^l,\\
\mbox{ and }\hskip3mm \beta(\sigma^a,\sigma^b)=\lambda(y^d)=\zeta_{2^s},  & &
\beta'(\sigma^a,\sigma^b)=\lambda'(y^d)=\zeta_{2^s}^{j}  \hskip5mm \mbox{ if } a+b\ge 2^l.
\end{align*}
As Gal$(E/F)$ is a cyclic extension, then the crossed  product algebra $(E, {\rm Gal}(E/F),\beta)$ is  isomorphic to the cyclic algebra $(E/F, \sigma, \zeta_{2^s})$ (see Section \ref{Section 4}). Similarly, the crossed  product algebra $(E, {\rm Gal}(E/F),\beta^{'})$ is  isomorphic to the cyclic algebra $(E/F, \sigma, \zeta^{j}_{2^s})$ (see Section \ref{Section 4}). From the Subsection \ref{subsection 3.4} (ii), if $s \geq 2$ then the cyclic algebras $(E/F, \sigma, \zeta_{2^s})$ and $(E/F, \sigma^{'}, \zeta^{j}_{2^s})$ are isomorphic to the matrix algebra $M_{2^l}(F)$. For $s = 0, 1$, it is clear that the cyclic algebras $(E/F, \sigma, \zeta_{2^s})$ and $(E/F, \sigma^{'}, \zeta^{j}_{2^s})$ are isomorphic to the matrix algebra $M_{2^l}(F)$. Thus, $(E, {\rm Gal}(E/F),\beta)$ and $(E, {\rm Gal}(E/F),\beta')$ are of the same index $1$. Hence, $m_F(\rho)=m_F(\rho') = 1$.

    {\it Subcase II. $F$ does not contain $\sqrt{-1}$.} As $p = 2$ and $F$ does not contain $\sqrt{-1}$, then ${\rm Gal}(E/F) = {\rm Gal}(F(\zeta_{2^{n}})/F)$ is  either cyclic group $C_{2}$ or non-cyclic abelian $2$-group.\\
If $G/K \cong {\rm Gal}(E/F)$ is cyclic group $C_{2}$, then $G$ is isomorphic to $Q_{2^{n}}, n \geq 3$, $D_{2^{n}}, n \geq 2$ or $SD_{2^{n}}, n \geq 4$ (see Section \ref{Section 4}). In all these three cases, from Lemma \ref{Lemma 4.1.1}, Lemma \ref{Lemma 4.1.2}, Lemma \ref{Lemma 4.1.3} and Remark \ref{Remark 4.1.4}, one can see that $m_F(\rho) = m_F(\rho')$, and the possibilities are $1$ or $2$.

Now consider ${\rm Gal}(E/F) = {\rm Gal}(F(\zeta_{2^{n}})/F)$ is a non-cyclic abelian $2$-group, in fact,
%Gal$(\mathbb{Q}(\zeta_{2^{n}})/\mathbb{Q})$ is isomorphic to $\mathbb{Z}_{2} \times \mathbb{Z}_{2^{n-2}}$ and for $n = 2$, ${\rm Gal}(\mathbb{Q}(\zeta_{2^{n}})/\mathbb{Q})$ is isomorphic to $\mathbb{Z}_{2}$.
isomorphic to $\mathbb{Z}_{2} \times \mathbb{Z}_{2^{k}}$, for some
positive integer $k$. Therefore,  $G/K$ is isomorphic to
$\mathbb{Z}_{2} \times \mathbb{Z}_{2^{k}}$. So $G$ turns out to be a
$2$-group which contains a maximal abelian normal
subgroup which is cyclic and $G/K$ is isomorphic to $\mathbb{Z}_{2} \times
\mathbb{Z}_{2^{k}}$. From Lemma \ref{Lemm 5.2}, $G$ is isomorphic to
one of the following two types of groups:\\
$P_{2} = \langle{a, b, c ~|~ a^{2^n} = b^{2^k} = c^2 = 1, bc = cb, b^{-1}ab = a^{r}, c^{-1}ac = a^{-1}}\rangle, with ~r \equiv 1~ {\rm mod}~ 4,$\\
$P_{3} = \langle{a, b, c ~|~ a^{2^n} = b^{2^k} = 1, c^2 = a^{2^{n-1}}, bc = cb, b^{-1}ab = a^{r}, c^{-1}ac = a^{-1}}\rangle, with~ r~ \equiv~ 1~{\rm mod}~ 4.$\\
As $P_{2} \cong C_{2^n} \rtimes (C_{2^k} \times C_{2})$, then the
Schur index of any faithful irreducible
$\overline{F}$-representation of $P_{2}$ w.r.t. $F$ is equal to $1$
(see \cite{Isaacs-1976}, Lemma 10.8). For $P_{3}$, from the Lemma \ref{Lemma 5.5}, the Wedderburn component corresponding to any faithful irreducible $\mathbb{Q}$-representation of $P_{3}$ is $M_{2^{k}}(\mathbb{H}(\mathbb{Q}(\zeta_{2^{n-k}} + \zeta^{-1}_{2^{n-k}})))$, where $\mathbb{H}(\mathbb{Q}(\zeta_{2^{n-k}} + \zeta^{-1}_{2^{n-k}}))$ is the quaternion algebra over $\mathbb{Q}(\zeta_{2^{n-k}} + \zeta^{-1}_{2^{n-k}})$; moreover this is a division algebra. So the Schur index of any faithful irreducible $\overline{F}$-representation of $P_{3}$ w.r.t $\mathbb{Q}$ is equal to $2$. As $P_{3}$ is a metabelian $2$-group, the character fields of all the faithful irreducible $\overline{F}$-representation of $P_{3}$ over $F$ are the same. Hence from the Remark \ref{Remark 4.0.1}, the Schur indices of all the faithful irreducible $\overline{F}$-representation of $P_{3}$ w.r.t. $F$ are the same, and is equal to always $1$ or $2$. Thus in this case also $m_F(\rho) = m_F(\rho')$. This completes the proof of statement $(1)$.

{\bf Proof of Theorem \ref{main} (2).} Now we prove the statement $(2)$. By
Theorem $(70.15)$ in \cite{Curtis-1962}, let $\rho\otimes_F
\overline{F}\cong m_F(\rho_1)(\rho_1\oplus \cdots \oplus
\rho_{\delta})$ and $\rho'\otimes_F \overline{F}\cong
m_F(\rho'_1)(\rho'_1\oplus \cdots \oplus \rho'_{\delta'})$. Since
$\rho_1$ and $\rho_1'$ are faithful irreducible
$\overline{F}$-representations, they are induced from one
dimensional $\overline{F}$-representations of $K$ (see Remark
\ref{remark3.3}), hence $\deg(\rho_1)=\deg(\rho_1')=[G:K]$. Let
$\chi_1$ and $\chi_1'$ be the characters of $\rho_1$ and $\rho_1'$
respectively. From the proof of $(1)$, we have
$F(\chi_1)=F(\chi_1')$. Hence
$\delta=[F(\chi_1):F]=[F(\chi_1'):F]=\delta'$. From part $(1)$, we
have $m_F(\rho_1)=m_F(\rho_1')$.  It follows that
$\deg(\rho)=\deg(\rho')$. This proves the statement $(2)$.

{\bf Proof of Theorem \ref{main} (3).}
Proof of the statement $(3)$ follows from the proof of statement $(1)$. More precisely, we have seen that in the proof of statement $(1)$, for $p$ odd, $m_F(\rho) = m_F(\rho') = 1$ and for $p = 2$, in all the subcases, $m_F(\rho) = m_F(\rho')$, and which is equal to $1$ or $2$. This proves the statement $(3)$.

{\bf Proof of Theorem \ref{main} (4).} Let $\rho$ be a faithful
irreducible $F$-representation of $G$ and $\eta$ be an irreducible
constituent of $\rho\downarrow_K^G$. By Lemma \ref{Lemma 2.0.1},
$\eta$ factors through faithful irreducible $F$-representation of
cyclic quotient $K/\textrm{ker}(\eta)$. As $\rho$ is faithful then
ker$(\eta)$ is core-free in $G$. Applying Theorem \ref{thm4.2}, and
noting that $m=1$ or $2$ (which is a consequence of the statement
$(3)$), we deduce that $\eta$ or $2(\eta)$ extends to an irreducible
$F$-representation of $I$ (the inertia subgroup of $\eta$ in $G$),
say, $\theta$ and $\rho \cong \theta\uparrow_I^G$. Therefore, the
degree of any faithful irreducible representation of $G$ over $F$ is
equal to $[G:I]\textrm{deg}(\eta)$ or $2[G:I]\textrm{deg}(\eta)$.
From the Lemma \ref{lemma3.1}, the degrees of irreducible
representations of $K$ over $F$ with core-free kernel in $G$ are the
same, say, $d$. From the Lemma \ref{Lemma 2.0.2}, $d$ is equal to
the common degree of irreducible factors of $p^{n}$-th cyclotomic
polynomial in $F[X]$. This completes the proof of the theorem.
\hfill$\square$

\begin{cor}
Let $G$ be a metabelian $p$-group and $F$ a field of characteristic $0$. Then all the faithful irreducible $F$-representations of $G$ (if they exist) have the same degree.
\end{cor}
\section{The Wedderburn Components of Faithful Irreducible Representations} \label{Section 7}
In this section, for a metabelian group $G$ with a prime power order maximal abelian normal subgroup containing $G^{'}$ and $F$ a field of characteristic $0$, we explicitly determine the Wedderburn components in the group algebra $F[G]$ corresponding to the faithful irreducible $\overline{F}$-representations.

In order to determine the Wedderburn components, we apply Theorem \ref{Theorem 7.0.1} of Yamada and the following two theorems.
\begin{thm}\label{Theorem 7.0.2}
Let $p$ be an odd prime and $F$ a field of characteristic $0$. Let $G$ be a metabelian group with a cyclic maximal abelian normal subgroup $K$ of order $p^n$, $n$ a positive integer, cotaining $G^{'}$. Let $\rho$ be a faithful irreducible $\overline{F}$-representation of $G$ with the character $\chi$. Let $\lambda$ be a  (faithful) irreducible $\overline{F}$-representation of the restriction $\rho \downarrow^{G}_{K}$. Let $I:=\{ g\in G : \mbox{there exists}~\tau(g)\in {\rm Gal}(F(\lambda)/F)~ \mbox{such that}~ \lambda^g=\lambda^{\tau(g)}\}.$ The following properties hold.
\begin{enumerate}
\item[(1)] $\rho = \lambda \uparrow^{G}_{K}$, $F(\lambda) = F(\zeta_{p^{n}})$ and $I/K \cong $ Gal$(F(\lambda)/F(\chi))$ is a cyclic group.
\item[(2)] Suppose that $I = G$. Let $G/K = \langle yK \rangle \cong C_{d}$, for some $y \in G$. Then there exists an automorphism $\sigma$ of $F(\lambda)$ such that $\lambda^y = \lambda^{\sigma}$. Moreover, the Wedderburn component corresponding to $\rho$ in $F[G]$ is isomorphic to $M_{d}(F(\chi)),$ and furthermore $F(\chi)$ is the subfield of $F(\lambda)$ fixed by $\sigma$.
\end{enumerate}
\end{thm}
\begin{proof}
As $\rho$ is a faithful irreducible $\overline{F}$-representation of $G$, and $\lambda$ an irreducible constituent of the restriction
$\rho\downarrow^{G}_{K}$, this implies that $\lambda$ is a
faithful degree one $\overline{F}$-representation of $K$. It is easy to see that the inertia subgroup of $\lambda$ in $G$ is $K$. Therefore by Clifford's theorem (see \cite{Curtis-1962}, Theorem $(49.7)$), $\rho \cong \lambda \uparrow^{G}_{K}$. Since $\lambda$ is a faithful irreducible $\overline{F}$-representation of $K$, $F(\lambda) = F(\zeta_{p^n})$. Once again, we use Theorem \ref{Theorem 7.0.1} of Yamada to deduce $I/K \cong $
Gal$(F(\lambda)/F(\chi)).$
As $p$ is odd, then Gal$(F(\zeta_{p^n})/F)$ is a cyclic group. Therefore
Gal$(F(\zeta_{p^n})/F(\chi))$ is also cyclic. This
completes the proof of $(1)$.

By assumption $I = G$. From
the part (1), we have $I/K \cong $ Gal$(F(\lambda)/F(\chi)),$ and so
$G/K = \langle yK \rangle \cong$ Gal$(F(\lambda)/F(\chi))$ $ = $
Gal$(F(\zeta_{p^n})/F(\chi))$ is a cyclic group of order $d$. Since for any $g \in G$, $\lambda^{g} =
\lambda^{\tau(g)}$, for some $\tau(g) \in$ Gal$(F(\lambda)/F)$, there exists a
an $F$-automorphism $\sigma$ of $F(\lambda)$ such that
$\lambda^y = \lambda^{\sigma}$. From the proof of the main theorem
Theorem \ref{main}, the Schur index of $\rho$ w.r.t. $F$ is equal to
$1$. Then the Wedderburn component corresponding to $\rho$ in $F[G]$
is $M_{d}(F(\chi))$. Now we show that $F(\chi)$ is the subfield of $F(\lambda)$ fixed by $\sigma$.

Note that $G = \langle{K, y}\rangle$, where $y^d \in K$. For each
$i$, $0 \leq i \leq d-1$, let $\lambda_{i}$ be the representation
conjugate to $\lambda$, and is defined by $\lambda_{i}(x) =
\lambda(y^{-i}xy^{i}),~ \textrm{for~all}~ x \in K.$ Note that the
representations $\{\lambda_{i} : 0 \leq i \leq d-1\}$ are
inequivalent, and $\rho = \lambda_{i} \uparrow^{G}_{K}$. From the
Frobenius reciprocity theorem (see \cite{Curtis-1962}, Theorem
$(38.8)$), these are the only irreducible representations of $K$
which induce $\rho$. Since $K$ is a normal subgroup of $G$, $\chi$
vanishes on $G - K$. Again since $\rho\downarrow^{G}_{K} =
\oplus^{d}_{i = 1} \lambda_{i}$, we have
$$\chi(x) = \sum^{d - 1}_{i = 0} \lambda_{i}(x) = \sum^{d - 1}_{i = 0} \lambda(y^{-i}xy^{i}), \textrm{for~all}~ x \in K.$$

Since the map $\lambda \mapsto \lambda^y$ is induced by the automorphism $\sigma$ of $F(\lambda)$, we have $\lambda(y^{-1}xy) = (\lambda(x))^{\sigma},~ \textrm{for~all}~ x \in K.$ Then $\textrm{for~all}~ x \in K$,
$$\chi(x) = \sum^{d-1}_{i = 0}{(\lambda(x))^{\sigma^i}}.$$
Therefore, for all $x \in K$, $\chi(x)$ belong to the field, which is fixed by $\sigma$. Since $\lambda(y^{-d}xy^{d}) = \lambda(x),$ $\textrm{for~all}~ x \in K$, we have $o(\sigma) = d$. As $F(\lambda) / F$ is cyclic, $[F(\lambda) : F(\chi)] = d$ and $o(\sigma) = d$,  $F(\chi)$ is the subfield of $F(\lambda)$ fixed by $\sigma$. This proves the theorem.
\end{proof}

\begin{thm} \label{Theorem 7.0.3}
Let $F$ be a field of characteristic $0$ containing $\sqrt{-1}$. Let $G$ be a metabelian group with a cyclic maximal abelian normal subgroup $K$ of order $2^n$, $n$ a positive integer, containing $G^{'}$. Let $\rho$ be a faithful irreducible $\overline{F}$-representation of $G$ with the character $\chi$. Let $\lambda$ be a (faithful) irreducible $\overline{F}$-representation of the restriction $\rho \downarrow^{G}_{K}$. Let $I:=\{ g\in G : \mbox{there exists}~\tau(g)\in {\rm Gal}(F(\lambda)/F)~ \mbox{such that}~ \lambda^g=\lambda^{\tau(g)}\}.$ The following properties hold.
\begin{enumerate}
\item[(1)] $\rho = \lambda \uparrow^{G}_{K}$, $F(\lambda) = F(\zeta_{2^{n}})$ and $I/K \cong $ Gal$(F(\lambda)/F(\chi))$ is a cyclic $2$-group.
\item[(2)] Suppose that $I = G$. Let $G/K = \langle yK \rangle \cong C_{2^l}$, for some $y \in G$. Then there exists an automorphism $\sigma$ of $F(\lambda)$ such that $\lambda^y = \lambda^{\sigma}$. Moreover, the Wedderburn component corresponding to $\rho$ in $F[G]$ is isomorphic to $M_{2^{l}}(F(\chi)),$ $F(\chi)$ is the subfield of $F(\lambda)$ fixed by $\sigma$. Furthermore, $F(\chi)$ is of the form $F(\zeta_{2^m})$, where $m \geq 2$.
\end{enumerate}
\end{thm}
\begin{proof}
The proof is similar to the proof of Theorem \ref{Theorem 7.0.2}, and we skip it.
\end{proof}
\begin{remark}
Let $G$ be a metabelian group with a prime power order maximal abelian normal subgroup $K$ containing $G^{'}$ and $F$ a field of characteristic $0$. Let exp$(K) = p^n$, for some positive integer $n$. Let $\rho$ be a faithful irreducible $\overline{F}$-representation of $G$ with the character $\chi$. Let $\lambda$ be an irreducible constituent of $\rho \downarrow^{G}_{K}$.
Let $I:=\{ g\in G\,: \,  \mbox{there exists}~\tau(g)\in {\rm Gal}(F(\lambda)/F)~ \mbox{such that}~ \lambda^g=\lambda^{\tau(g)}\}$.
Let $A_{F}(\chi)$ be the Wedderburn component of $F[G]$ corresponding to $\chi$.
\begin{itemize}
\item[(1)] If $p$ is odd, then from the proof of Theorem \ref{main} ($p$ odd case), Theorem \ref{Theorem 7.0.1} of Yamada and by using the Theorem \ref{Theorem 7.0.2}, we can explicitly determine $A_{F}(\chi)$.

\item[(2)] If $p = 2$ and $F$ contains $\sqrt{-1}$, then from the proof of Theorem \ref{main} ($p = 2$ case), Theorem \ref{Theorem 7.0.1} of Yamada and by using the Theorem \ref{Theorem 7.0.3}, $A_{F}(\chi)$ can be determined.

\item[(3)] If $p = 2$, $\sqrt{-1} \notin F$, and $I/K \cong C_{2}$, then $G$ contains a section isomorphic to a generalized quaternion group $Q_{2^n}, n \geq 3$ or dihedral group $D_{2^{n}}, n \geq 2$ or semidihedral group $SD_{2^{n}}, n \geq 4$. In this case, once again by using Theorem \ref{Theorem 7.0.1}, Lemmata \ref{Lemma 4.1.1}, \ref{Lemma 4.1.2}, \ref{Lemma 4.1.3} and Remark \ref{Remark 4.1.4}, the Wedderburn component $A_{F}(\chi)$ can be obtained.

\item[(4)] If $p = 2$, $\sqrt{-1} \notin F$, and $I/K \cong C_{2} \times C_{2^k}$, then $G$ contains a section isomorphic to $P_2$ or $P_{3}$ (see Lemma \ref{Lemm 5.2}). In this case, by Theorem \ref{Theorem 7.0.1}, Lemmata \ref{Lemma 6.0.2}, \ref{Lemma 5.5} and from the Remark \ref{Remark 4.0.1}, we can explicitly determine $A_{F}(\chi)$.
\end{itemize}
\end{remark}
\section{An Example}\label{Section 8}
In this section, we give an example which illustrates the Theorem
\ref{main} (see also \cite{Pradhan-2021}). We shall use the notation of
the previous two sections.

Let $G = C_{7}\rtimes C_{3} = \langle{x, y ~|~ x^{7}  = y^{3} = 1, y^{-1}xy = x^2}\rangle$, a metacyclic group; thus a metabelian. Note that $G^{'} = \langle x \rangle \cong C_{7}$. We take $F$ to be $\mathbb{Q}$.
%$G$ has a unique minimal normal subgroup $N  = \langle{x \,|\, x^{7}  = 1}\rangle$, and so $G$ is a metacyclic monolith group.
Let $K = \langle{x \,|\, x^{7}  = 1}\rangle$, maximal abelian normal 7-subgroup of $G$ cotaining $G^{'}$. $G$ has two degree three faithful irreducible complex representations, and they are given by
$$[\rho(x)] = \begin{bmatrix}
\zeta & 0  & 0 \\
0 & \zeta^{2} & 0 \\
0 & 0 & \zeta^{4}\\
\end{bmatrix},
[\rho(y)] = \begin{bmatrix}
0 & 0  & 1\\
1 & 0 & 0\\
0 & 1 & 0\\
\end{bmatrix},$$
and
$$[\rho^{'}(x)] = \begin{bmatrix}
\zeta^3 & 0  & 0 \\
0 & \zeta^{6} & 0 \\
0 & 0 & \zeta^{5}\\
\end{bmatrix},
[\rho^{'}(y)] = \begin{bmatrix}
0 & 0  & 1\\
1 & 0 & 0\\
0 & 1 & 0\\
\end{bmatrix},$$
where $\zeta = \zeta_{7}$ is a primitive $7$-th root of unity. Let $\chi$ and $\chi^{'}$ be the characters corresponding to $\rho$ and $\rho^{'}$ respectively. Then the character fields of $\rho$ and $\rho^{'}$ over $F$ are: $F(\chi) = \mathbb{Q}(\zeta + \zeta^2+ \zeta^4)$ and $F(\chi^{'}) = \mathbb{Q}(\zeta^3 + \zeta^6+ \zeta^5)$ respectively. Note that $F(\chi) = F(\chi^{'})$. We compute $m_{F}(\rho)$ and $m_{F}(\rho^{'})$.
As $m_{F}(\rho) = m_{F(\chi)}(\rho)$, $m_{F}(\rho^{'}) = m_{F(\chi^{'})}(\rho^{'})$ and $F(\chi) = F(\chi^{'})$, so without loss of generality, we can take $F$ to be $\mathbb{Q}(\zeta + \zeta^2+ \zeta^4)$ for computing Schur indices.

Let $\lambda$ and $\lambda^{'}$ be two faithful irreducible complex representations of $K$, and are defined by $\lambda: x \mapsto \zeta,$ $ \lambda^{'}: x \mapsto \zeta^3.$
Observe that $\rho \cong \lambda\uparrow^{G}_{K}$, $\rho \downarrow^{G}_{K} = \lambda \oplus \lambda^{y} \oplus \lambda^{y^2}$; $\rho^{'} \cong \lambda^{'}\uparrow^{G}_{K}$, $\rho^{'} \downarrow^{G}_{K} = \lambda^{'} \oplus {\lambda^{'}}^{y} \oplus {\lambda^{'}}^{y^2}$ and $I = I^{'} = G$.

Let $E:= F(\lambda) = F(\lambda^{'}) = \mathbb{Q}(\zeta)$. Then $E/F$ is a cyclic extension with Gal$(E/F)=$ Gal$(\mathbb{Q}(\zeta)/\mathbb{Q}(\zeta + \zeta^2 + \zeta^4)) \cong C_{3}$. Consider $\sigma$ : $\zeta \mapsto \zeta^2$, which is an automorphism of $\mathbb{Q}(\zeta)$. Note that Gal$(E/F)= \langle \sigma \rangle.$ Observe that $\lambda^{y} = \lambda^{\sigma}$ and ${\lambda^{'}}^y = {\lambda^{'}}^{\sigma}$. The Wedderburn component corresponding to the irreducible $F$-representation $m_{F}(\rho)\rho$ is: $(E/F, \sigma, 1) \cong M_{3}(\mathbb{Q}(\zeta + \zeta^2 + \zeta^4))$. The Wedderburn component corresponding to the irreducible $F$-representation $m_{F}(\rho^{'})\rho^{'}$: $(E/F, \sigma, 1) \cong M_{3}(\mathbb{Q}(\zeta^3 + \zeta^6 + \zeta^5))$. So $m_{F}(\rho) = m_{F}(\rho^{'}) = 1$, and this implies that $m_{\mathbb{Q}}(\rho) = m_{\mathbb{Q}}(\rho^{'}) = 1$.

Since $\rho$, $\rho^{'}$ are algebraically conjugate over $\mathbb{Q}$, $G$ has unique faithful irreducible $\mathbb{Q}$-representation of degree $6$, and the corresponding Wedderburn component in $\mathbb{Q}[G]$ is $M_{3}(\mathbb{Q}(\zeta + \zeta^2 + \zeta^4))$.
\section{Solvable groups where main result fails}\label{Section 9}
In this final section, we mention two examples where the main
theorem of this paper do not carry over. The first is that of a
solvable group of derived length $>2$ and the second one is that of
a metabelian group which does not have a prime power order maximal abelian normal
subgroup containing $G^{'}$.
\begin{remark}
The statement $(1)$ and statement $(2)$ of Theorem \ref{main} may not be true for groups of derived length more than $2$. For example consider ${\rm SL}_2(\mathbb{F}_3)$, which is of derived length $3$.
%Note that $SL_{2}(\mathbb{F}_3)$ has a unique minimal normal subgroup of order $2$, implies that $SL_{2}(\mathbb{F}_3)$ is a monolithic group of derived length $3$.
${\rm SL}_2(\mathbb{F}_3)$ has three degree 2 faithful irreducible complex representations, say, $\rho_5, \rho_6, \rho_7$ with characters $\chi_5, \chi_6, \chi_7$ respectively (see \cite{Isaacs-1976}, p.$288$). Note that $\rho_5$ is not realizable over $\mathbb{Q}$, and in fact it is  realizable over $\mathbb{Q}(\sqrt{-1})$. So the Schur index of $\rho_5$ w.r.t. $\mathbb{Q}$ is $2$. On the other hand, the characters of $\rho_6$ and $\rho_7$ are Galois conjugate over $\mathbb{Q}$ and their Schur index w.r.t. $\mathbb{Q}$ is $1$.
Thus over $\mathbb{Q}$, $2(\rho_5)$ and $\rho_6 \oplus \rho_7$ are two faithful irreducible representations of ${\rm SL}_2(\mathbb{F}_3)$ of the same degree $4$. In this case, the statement $(1)$ does not hold true, but the statement $(2)$ holds true.

Now consider over $\mathbb{Q}(\sqrt{-1})$ instead of $\mathbb{Q}$. The Schur index of $\rho_5$ over $\mathbb{Q}(\sqrt{-1})$ is equal to $1$. On the other hand, the characters of $\rho_6$ and $\rho_7$ are Galois conjugate over $\mathbb{Q}(\sqrt{-1})$ and their Schur index w.r.t. $\mathbb{Q}(\sqrt{-1})$ is $1$. Hence $\rho_6 \oplus \rho_7$ is equivalent to a $\mathbb{Q}(\sqrt{-1})$-irreducible representation of ${\rm SL}_2(\mathbb{F}_3)$. Thus over $\mathbb{Q}(\sqrt{-1})$, $\rho_5$ and $\rho_6 \oplus \rho_7$ are faithful irreducible representations of ${\rm SL}_2(\mathbb{F}_3)$ of degree $2$ and $4$ respectively. In this case, statement $(1)$ holds true, but statement $(2)$ does not hold true. Though in both situations over $\mathbb{Q}$ and $\mathbb{Q}(\sqrt{-1})$, the Schur index of any faithful irreducible complex representation of ${\rm SL}_2(\mathbb{F}_3)$ is equal to $1$ or $2$.
\end{remark}
\begin{remark}
The statement $(3)$ of Theorem \ref{main} may not be true for a metabelian group which does not have any prime power order maximal abelian normal subgroup containing the derived subgroup. Let
$$G = \langle{x , y ~|~ x^7 = y^9 = 1, y^{-1}xy = x^2}\rangle.$$
Note that $G^{'} = \langle{x}\rangle \cong C_{7}$.
Let $K = \langle{x, y^3}\rangle \cong C_{21}$ be a maximal abelian normal subgroup of $G$ containing $G^{'}$. Let $F = \mathbb{Q}$. $G$ has four degree three faithful irreducible complex representations $\rho_{1}, \rho_{2}, \rho_{3}, \rho_{4}$, and  $\chi_{1}, \chi_{2}, \chi_{3}, \chi_{4}$ be their corresponding characters respectively.
Let $\lambda_{1}, \lambda_{2}, \lambda_{3}, \lambda_{4}$ be the irreducible complex representations of $K$, and are given by:
$\lambda_{1}(x) = \zeta,~ \lambda_{1}(y^3) = \omega;$
%It is clear that $\rho_{1} = \lambda_{1} \uparrow^{G}_{K}.$
%Similarly, $\lambda_{2}, \lambda_{3}, \lambda_{4}$ are defined by
$\lambda_{2}(x) = \zeta^{3},~ \lambda_{2}(y^3) = \omega^2;$
$\lambda_{3}(x) = \zeta,~ \lambda_{3}(y^3) = \omega^2;$
$\lambda_{4}(x) = \zeta^{3},~ \lambda_{4}(y^3) = \omega,$
and $\rho_{1} = \lambda_{1} \uparrow^{G}_{K}, \rho_{2} = \lambda_{2} \uparrow^{G}_{K}, \rho_{3} = \lambda_{3} \uparrow^{G}_{K}, \rho_{4} = \lambda_{4} \uparrow^{G}_{K}.$
The character fields of $\chi_{1}, \chi_{2}, \chi_{3}, \chi_{4}$ over $F$ are: $F(\chi_{1}) = \mathbb{Q}(\zeta + \zeta^2+ \zeta^4, \omega), F(\chi_{2}) = \mathbb{Q}(\zeta^3 + \zeta^6+ \zeta^5, \omega^2)$, $F(\chi_{3}) = \mathbb{Q}(\zeta + \zeta^2+ \zeta^4, \omega^2), F(\chi_{4}) = \mathbb{Q}(\zeta^3 + \zeta^6+ \zeta^5, \omega)$. Note that $F(\chi_{1}) = F(\chi_{2}) = F(\chi_{3}) = F(\chi_{4})$, and is of degree $4$ extension of $\mathbb{Q}$. So $\chi_{1}, \chi_{2}, \chi_{3}, \chi_{4}$ are Galois conjugates over $\mathbb{Q}$. From the Theorem 10.16 in \cite{Isaacs-1976}, the Schur index of $\rho_{i}~(i = 1,2,3,4)$ w.r.t. $\mathbb{Q}$ is equal to $3$. Thus, $G$ has unique faithful irreducible $\mathbb{Q}$-represtation of degree $36$.
\end{remark}
{\bf Acknowledgement:}
The work is done while the first author is post doctoral fellow at Harish-Chandra Research Institute (HRI), Prayagraj (Allahabad), India. He wishes to express thanks to the institute for giving all the facilities to complete this work. The first author would like to thank Rahul Dattatraya Kitture for useful discussions.

\end{document}